\documentclass[12pt]{amsart}

\usepackage[matrix,arrow,curve,frame]{xy}
\usepackage{amsmath,amsthm,amssymb,enumerate}
\usepackage{latexsym}
\usepackage{amscd}
\usepackage[colorlinks=false]{hyperref}
\usepackage{euscript}

\newtheorem{thm}[subsection]{Theorem}
\newtheorem{defn}[subsection]{Definition}
\newtheorem{prop}[subsection]{Proposition}

\newtheorem{lemma}[subsection]{Lemma}

\theoremstyle{definition}

\numberwithin{equation}{section}

\def\bra{\langle}
\def\ket{\rangle}
\def\cS{{\mathcal S}}
\def\bS{{\mathcal S_+}}
\def\cE{{\mathcal E}}

\def\cW{{\mathcal W}}
\def\bW{{\mathcal W_+ }}

\def\cA{{\cal A}}

\def\cA{{\mathcal A}}

\def\cE{{\mathcal E}}

\def\cL{{\mathcal L}}

\def\cN{{\mathcal N}}

\DeclareMathOperator{\Sym}{Sym}
\DeclareMathOperator{\End}{End}

\newfont{\german}{eufm10}

\begin{document}
\pagestyle{plain}

\title{ The global sections of chiral de Rham complexes on compact Ricci-flat K\"ahler manifolds II}

\author{Andrew R. Linshaw}
\address{Department of Mathematics, University of Denver}
\email{andrew.linshaw@du.edu}
\thanks{A. Linshaw is supported by Simons Foundation Grant \#635650 and NSF Grant DMS-2001484.}

\author{Bailin Song}
\address{School of Mathematical Sciences, University of Science and Technology of China, Hefei, Anhui 230026, P.R. China}
\email{bailinso@ustc.edu.cn}

\thanks{B. Song is supported by the National Natural Science Foundation of China No. 12171447}

\begin{abstract} We give a complete description of the vertex algebra of global sections of the chiral de Rham complex of an arbitrary compact Ricci-flat K\"ahler manifold.
\end{abstract}
\maketitle
\section{Introduction} The chiral de Rham complex $\Omega^{\text{ch}}_X$ is a sheaf of vertex superalgebras that exists on any smooth manifold $X$ in either the algebraic, complex analytic or smooth settings. It is bigraded by degree and conformal weight, and contains the ordinary de Rham sheaf as the weight zero component. The de Rham differential extends to a square-zero differential on the entire structure which preserves conformal weight and raises the degree by one. This sheaf was introduced by Malikov, Schechtman and Vaintrob in \cite{MSV}, and has attracted significant attention in both the physics and mathematics literature. The space of global sections $\Gamma(X, \Omega^{\text{ch}}_X)$ is always a vertex superalgebra, and it is known to have extra structure when $X$ is endowed with geometric structures. For example, if $X$ has a Riemannian metric, it has an $\cN=1$ superconformal structure, and when $X$ is K\"ahler or hyperk\"ahler this is enhanced to an $\cN=2$ structure and $\cN=4$ structure, respectively \cite{BHS}. In these cases, certain covariantly closed differential forms on $X$ also give rise to fields in $\Gamma(X, \Omega^{\text{ch}}_X)$. For example, when $X$ is Calabi-Yau it was shown in \cite{EHKZ} that $\Gamma(X, \Omega^{\text{ch}}_X)$ contains a subalgebra generated by $8$ fields that was introduced by Odake \cite{O}.

Until recently, a complete description of $\Gamma(X, \Omega^{\text{ch}}_X)$ was not known in any examples other than an affine space or a torus. In \cite{D}, for any congruence subgroup $G \subset SL(2,\mathbb{R})$, Dai constructed a basis of the $G$-invariant global sections of the chiral de Rham complex on the upper half plane, which are holomorphic at the cusps. The vertex operations are determined by a modification of the Rankin-Cohen brackets of modular forms. In \cite{S2}, the second author showed that for a compact Ricci-flat K\"ahler manifold with holonomy group $SU(d)$ or $Sp(\frac d 2)$, $\Gamma(X,\Omega^{\text{ch}}_X)$ is isomorphic to a certain subalgebra of the $bc\beta\gamma$-system of rank $d = \text{dim}\ X$ which is invariant under the action of an infinite-dimensional Lie algebra of Cartan type. An explicit description of this invariant space was conjectured in \cite{S2} and this conjecture was proven in the case $d = 2$ using results on the invariant theory of arc spaces developed in \cite{LSSI}. This allowed a complete description of $\Gamma(X, \Omega^{\text{ch}}_X)$ for all K3 surfaces; it is isomorphic to the simple (small) $\cN=4$ superconformal algebra with central charge $c = 6$ \cite{S,S1}.

Very recently, in a series of papers \cite{LSI,LSII,LSIII} we have proven the arc space analogues of the first and second fundamental theorems of invariant theory for the general linear, special linear, and symplectic groups. This was achieved by providing a standard monomial basis for these invariant spaces that extends the standard monomial basis in the classical setting. These results provide the needed ingredients to complete the description of $\Gamma(X, \Omega^{\text{ch}}_X)$ for a general compact Ricci-flat K\"ahler manifold. Unfortunately, this approach does not generalize to K\"ahler manifolds which are not Ricci-flat, since there is no method to describe the global sections of tensor powers of the tangent and cotangent bundles. In general, $\Gamma(X, \Omega^{\text{ch}}_X)$ need not be isomorphic to a subalgebra of a free field algebra which is invariant under a Lie algebra of Cartan type.

The plan of the paper is following. In Section \ref{sec:two}, we introduce the $\beta\gamma-bc$ system. In Section \ref{sec:2.5}, we introduce the Lie algebras of Cartan type and their actions on the $\beta\gamma-bc$ systems. In Section \ref{sec:foure}, we calculate the subspaces of invariant elements in $\beta\gamma-bc$ systems under the action of special series and Hamiltonian series of Lie algebras of Cartan type, by reducing this to the invariant theory of arc spaces. Finally, in Section \ref{sec:three}, we calculate the space of global sections of the chiral de Rham complexes on compact Ricci-flat K\"ahler manifolds.

\section{$\beta\gamma-bc$ system}\label{sec:two}
\subsection{Vertex algebras}
In this paper, we will follow the formalism of vertex algebras developed in \cite{Kac}. A vertex
algebra is the data $(\cA, Y, L_{-1}, 1)$. In this notation,
\begin{enumerate} 
\item $\cA$ is a $\mathbb Z_2$-graded vector space over $\mathbb C$. The $\mathbb Z_2$-grading is called parity, and $|a|$ denotes the parity of a homogeneous element $ a \in \cA$. 
\item $Y$ is an even linear map
$$Y : \cA \to \End (\cA) [[z, z^{-1}]],\qquad Y(a) = a(z) = \sum_{n\in \mathbb Z}a_{(n)}z^{-n-1}.$$ Here $z$ is a formal variable and $a(z)$ is called the field corresponding to $a$.  
\item $1 \in \cA$ is called the vacuum vector. 
\item $L_{-1}$ is an even endomorphism of $\cA$.
\end{enumerate}
They
satisfy the following axioms:
\begin{itemize}
	\item
	\textsl{Vacuum axiom.} $L_{-1}1 =0$; $1(z)= Id$; for $a\in \cA$, $n\geq 0$, $a_{(n)}1=0$ and $a_{(-1)}1=a$;
	\item \textsl{Translation invariance axiom.} For $a\in \cA$, $[L_{-1}, Y(a)] =\partial a(z)$;
	\item \textsl{Locality axiom.} Let $z,w$ be formal variables. For homogeneous $a, b \in \cA$, $(z-w)^k [a(z), b(w)]=0$ for some $k\geq 0$, where $[a(z), b(w)] = a(z) b(w) - (-1)^{|a| |b|} b(w) a(z)$.
\end{itemize}
For $a,b\in \cA$, $n\in \mathbb Z_{\geq 0}$, $a_{(n)}b$ is their $n^{\text{th}}$ product and their operator product expansion (OPE) is
$$a(z)b(w)\sim \sum_{n\geq 0} (a_{(n)}b)(w)(z-w)^{-n-1}.$$
The Wick product of $a(z)$ and $b(z)$ is
$:a(z)b(z):  \ =(a_{(-1)}b)(z)$. The other negative products are given by
$$:\partial^na(z)b(z):\ =n!(a_{(-n-1)}b)(z).$$
For $a_1, \dots , a_k\in \cA$, their iterated Wick product is defined to be
$$:a_1(z)\cdots a_k(z):\ =\ :a_1(z)b(z):,\quad \quad b(z)=\ :a_2(z)\cdots a_k(z):.$$
We often omit the formal variable $z$ when no confusion can arise.

We say that $\cA$ is generated by a subset $\{\alpha^i|\ i\in I\}$ if $\cA$ is spanned by all words in the letters $\alpha^i$, and all products, for $i\in I$ and $n\in\mathbb{Z}$. We say that $\cA$ is strongly generated by $\{\alpha^i|\ i\in I\}$ if $\cA$ is spanned by words in the letters $\alpha^i$, and all products for $n<0$. Equivalently, $\cA$ is spanned by the monomials $$\{ :\partial^{k_1} \alpha^{i_1}\cdots \partial^{k_m} \alpha^{i_m}:| \ i_1,\dots,i_m \in I,\ k_1,\dots,k_m \geq 0\}.$$

 For $a,b\in \cA$, the following identities will be frequently used.
\begin{eqnarray}\label{eq:circp}
	:ab:_{(n)}&=&\sum_{k<0}a_{(k)}b_{(n-k-1)}+(-1)^{|a||b|}\sum_{k\geq 0}b_{(n-k-1)}a_{(k)},\\
	\label{eq:circpcom} a_{(n)}b&=&\sum_{k\in \mathbb Z}(-1)^{k+1}(-1)^{|a||b|}(b_{(k)}a)_{(n-k-1)}1.
\end{eqnarray}

\subsection{$\beta\gamma-bc$ system}
Let $V$ be a $d$-dimensional complex vector space.
The $\beta\gamma$-system  $\cS(V)$ and $bc$-system $\cE(V)$ were introduced in \cite{FMS}.
The $\beta\gamma$-system  $\cS(V)$ is strongly generated by even elements $\beta^{x'}(z)$, $x' \in V$ and
$\gamma^{x}(z),x\in V^*$. The nontrivial OPEs among these generators are
$$ \beta^{x'}(z)\gamma^{x}(w)\sim {\bra x,x'\ket}{(z-w)}^{-1}.$$
The $bc$-system $\cE(V)$ is strongly generated by odd elements $b^{x'}(z)$, $x'\in V$ and
$c^x(z),x\in V^*$. The nontrivial OPEs among these generators are
$$ b^{x'}(z)c^{x}(w)\sim {\bra x,x'\ket}{(z-w)}^{-1}.$$
Here for $P=\beta,\gamma,b$ or $c$, we assume $a_1 P^{x_1}+a_2P^{x_2}=P^{a_1x_1+a_2x_2}$.

Let $$\cW(V):=\cS(V)\otimes\cE(V).$$
Let $\alpha^{x}=\partial \gamma^{x}$. Then $\beta^{x'}$ and $\alpha^{x}$ satisfy
$$ \beta^{x'}(z)\alpha^{x}(w)\sim {\bra x,x'\ket}{(z-w)}^{-2}.$$
Let $\bS (V)$ be the subalgebra of $\cS(V)$ generated by $\beta^{x'}$ and $\alpha^{x}$, so that $\bS(V)$ is a system of $2d$ free bosons. Let
$$\bW(V):=\bS(V)\otimes \cE(V).$$

If $V'$ is a vector space and $\psi:V\to V'$ is a linear isomorphism, let 
$\psi^*: V'^*\to V^*$ be the induced map on dual spaces. Then $\psi$ induces an isomorphism of vertex algebras 
\begin{equation} \begin{split} & \cW(\psi):\cW(V)\to \cW(V'),
\\ & \beta^{x'}\mapsto \beta^{\psi(x')},\qquad b^{x'}\mapsto b^{\psi(x')},\qquad \gamma^x\mapsto \gamma^{(\psi^*)^{-1}(x)},\qquad c^x\mapsto c^{(\psi^*)^{-1}(x)}.\end{split} \end{equation}
Note that $\cW(\psi)$ restricts to an isomorphism $\cW_+(V)\cong \cW_+(V')$. 

Fix $x'_1,\dots, x'_d$, a basis of $V$ and let $x_1,\dots, x_d$ be the dual basis of $V^*$.
Let $S_0$ be the set of $\beta_{(n)}^{x'_i},\alpha^{x_i}_{(n)}, b_{(n)}^{x'_i}, c_{(n)}^{x_i}$, $1\leq i\leq d$, $n< 0$.
These operators are supercommutative.
Let $SW(V)=\mathbb C[S_0]$ be the algebra generated by these operators. There is a canonical isomorphism of $SW(V)\otimes_{\mathbb C}\mathbb C[\gamma^{x_1}_{(-1)}] $ modules,
\begin{equation*} \tilde \pi: SW(V)\otimes_{\mathbb C}\mathbb C[\gamma^{x_1}_{(-1)},\dots, \gamma^{x_d}_{(-1)}]  \to  \mathcal W(V), \qquad a \otimes f \mapsto af1. \end{equation*}
In particular, $\cW(V)$ is a free $\mathbb C[\gamma^{x_1}_{(-1)},\dots, \gamma^{x_d}_{(-1)}]$-module.
Restricting $\tilde \pi$ to $SW(V)\otimes \{1\}$, we get an isomorphism of $SW(V)$ modules,
\begin{equation}\label{eqn:isopi}
	\pi:SW(V)\to \bW(V),\quad a\mapsto a 1.
\end{equation}

\subsection{Subalgebras of $\cW_+(V)$}
Let
\begin{equation} \label{eqn:globlesection0}
\begin{split} & Q(z)= \sum_{i=1}^d:\beta^{x'_i}(z)c^{x_i}(z):, \qquad L(z)=\sum_{i=1}^d(:\beta^{x'_i}(z)\partial\gamma^{x_i}(z):-:b^{x'_i}(z)\partial c^{x_i}(z):),
\\ & J(z)=-\sum_{i=1}^d:b^{x'_i}(z)c^{x_i}(z):, \qquad G(z)=\sum_{i=1}^d:b^{x'_i}(z)\partial\gamma^{x_i}(z):,\end{split} \end{equation}
Note that $L$ is a Virasoro field in $\cW(V)$ of central charge zero, and $b^{x'_i}$, $c^{x_i}$, $\beta^{x'_i}$, $\gamma^{x'_i}$ are primary of weights $1,0,1,0$ with respect to $L$. Also, $J$ generates a Heisenberg algebra and the zero mode $J_{(0)}$ induces an additional $\mathbb{Z}$-grading called the degree; note that $b^{x'_i}$, $c^{x_i}$, $\beta^{x'_i}$, $\gamma^{x'_i}$ have degrees $-1,1,0,0$. Finally, we recall that $L$ can be replaced with the Virasoro field $T = L - \frac{1}{2} \partial J$. This has central charge $c = 3d$, and $b^{x'_i}$, $c^{x_i}$, $\beta^{x'_i}$, $\gamma^{x'_i}$ are primary of weights $\frac{1}{2},\frac{1}{2},1,0$ with respect to $T$. The subalgebra of $\cW_+(V)$ generated by $Q, T,J, G$ (equivalently, $Q, L, J, G$) is isomorphic to the $\cN = 2$ superconformal algebra with central charge $c = 3d$.

Next, let
\begin{equation} \label{eqn:globlesection1}
\begin{split} & D(z)= \ : b^{x'_1}(z)b^{x'_2}(z)\cdots b^{x'_d}(z):, \qquad E(z)= \ :c^{x_1}(z)c^{x_2}(z)\cdots c^{x_d}(z):,
\\ & B(z)=Q(z)_{(0)}D(z), \qquad \qquad \qquad \quad C(z)=G(z)_{(0)}E(z).\end{split} \end{equation}

If $d=2l$ is even, let 
\begin{equation} \label{eqn:globlesection2}
\begin{split} & D'(z)= \sum_{i=1}^l : b^{x'_{2i-1}}(z)b^{x'_{2i}}(z):,\qquad E'(z)=\sum_{i=1}^l : c^{x_{2i-1}}(z)c^{x_{2i}}(z):,
\\ & B'(z)=Q(z)_{(0)}D'(z),\quad\quad\quad\quad\quad \quad C'(z)=G(z)_{(0)}E'(z).\end{split} \end{equation}
\begin{defn} Let $\cA_0(V)$  be the vertex algebra generated by the fields \eqref{eqn:globlesection0} and \eqref{eqn:globlesection1}. Let $\cA_1(V)$ be the vertex algebra generated by the fields \eqref{eqn:globlesection0} and \eqref{eqn:globlesection2}.\end{defn}
The algebra $\cA_0(V)$ was introduced by Odake in \cite{O} and was studied extensively in the case $d=3$. It is easy to verify that the fields  \eqref{eqn:globlesection0} and \eqref{eqn:globlesection1} strongly generate $\cA_0(V)$. Similarly, $\cA_1(V)$ is strongly generated by the fields \eqref{eqn:globlesection0} and \eqref{eqn:globlesection2}, and is isomorphic to the simple small $\mathcal N =4$ superconformal vertex algebra with central charge $c=3d$. In \cite{S2}, we have shown that $\bW(V)$ is a unitary representation of $\cA_0(V)$ and $\cA_1(V)$.

\section{Lie algebras of Cartan type and their action on $\beta\gamma-bc$ system}
\label{sec:2.5}
\subsection{Lie algebras of Cartan type}
The space of algebraic vector fields on $V$ is  a graded Lie algebra $$\text{Vect}(V)=\oplus_{n\geq -1}\text{Vect}_n(V),\quad \text{Vect}_n(V)=\Sym^{n+1}(V^*)\otimes V.$$
If $x_1,\dots, x_d$ is a basis of $V^*$, then any element $v\in \text{Vect}_n(V)$ can be written as $v=\sum_{i=1}^d P_i\frac{\partial}{\partial x_i}$, where $P_i$ is a homogeneous
polynomial of degree $n+1$. For
$\sum_{i=1}^d P_i\frac{\partial}{\partial x_i}\in \text{Vect}_n(V, \omega_0)$ and $\sum_{j=1}^d P'_j\frac{\partial}{\partial x_j}\in \text{Vect}_m(V)$,
$$[\sum_{i=1}^d P_i\frac{\partial}{\partial x_i},\sum_{j=1}^d P'_j\frac{\partial}{\partial x_j}]
= \sum_{i,j}( P_i\frac{\partial P'_j}{\partial x_i}\frac{\partial}{\partial x_j} -P'_j\frac{\partial P_i}{\partial x_j}\frac{\partial}{\partial x_i} )\in \text{Vect}_{n+m}(V).$$
This Lie algebra is called the \textsl{general series}.
For a $k$-form $\omega\in \wedge^kV^*$, let
\begin{eqnarray*}
	\text{Vect}_n(V,\omega)&=&\{v\in \text{Vect}_n(V)|L_v \omega=0\},\\
	\text{Vect}(V,\omega)&=&\bigoplus_{n\geq -1}\text{Vect}_n(V,\omega).
\end{eqnarray*}
Here $L_v$ is the Lie derivative of $v$. Note that $\text{Vect}(V,\omega)$ is a graded Lie subalgebra of $\text{Vect}(V)$. We now consider $\text{Vect}(V,\omega)$ for some particular choices of $\omega$.

\begin{enumerate}
\item If $\omega_0=dx_1\wedge\cdots\wedge dx_d$,
$$ \text{Vect}_n(V,\omega_0)
=\{\sum_{i=1}^d P_i\frac{\partial}{\partial x_i}\in \text{Vect}_n(V)|\sum \frac{\partial}{\partial x_i}P_i=0\}.$$ The Lie algebra  $\text{Vect}(V,\omega_0)$ is called the \textsl{special series}. $\text{Vect}_0(V,\omega_0)$ is a Lie algebra isomorphic to $\mathfrak{sl}_d(\mathbb C)$.
\item If $d=2l$ is even and $\omega_1=\sum_{i=1}^ldx_{2i-1}\wedge dx_{2i}$.
The Lie algebra $\text{Vect}(V,\omega_1)$ is called the \textsl{Hamiltonian series}, and $\text{Vect}_0(V,\omega_1)$ is a Lie algebra isomorphic to $\mathfrak{sp}_d(\mathbb C)$.
\item If  $d=2l+1$ and $\omega = dx_{2l+1}+\sum_{i=1}^l( x_{l+i}dx_i-x_idx_{l+i})$. The Lie algebra
$$\{v\in \text{Vect}(V)|L_v\omega=P\omega, P\in \Sym^*( V^*)\}$$ is called the \textsl{contact series}.
\end{enumerate}

The general series, special series, Hamiltonian series and contact series are called the Lie algebras of
Cartan type and constitute an important class of simple infinite dimensional Lie algebras. In this paper, we consider the
special series and Hamiltonian series. 

\subsection{The actions of Lie algebras of Cartan type on $\beta\gamma-bc$ systems}
$\text{Vect}(V)$ has a canonical action on $\mathcal W(V)$ according to the Part III of \cite{MS}. Let  $\mathcal L: \text{Vect}(V)\to \text{Der}(\cW(V))$ be the map given by
\begin{equation}\label{eqn:actionL}
	\mathcal L(\sum_iP_i(x_1,\dots, x_d)\frac{\partial}{\partial x_i})= \sum_i(Q_{(0)} :P_i(\gamma^{x_1},\dots, \gamma^{x_d})b^{x_i'}:)_{(0)}.
\end{equation}
Clearly $\mathcal L$ is a homomorphism of Lie algebras.

\section{$\text{Vect}(V, \omega_i)$-invariants}\label{sec:foure}
For $R\subset \cW(V)$, let $$R^{\text{Vect}(V, \omega_i)}=\{a\in R\ | \mathcal L(g) a=0, \text{for any }g\in \text{Vect}(V, \omega_i)\}$$ be the space of $\text{Vect}(V, \omega_i)$-invariants. In \cite{S2}, the second author has shown that 
\begin{lemma}\label{lem:Linvariant}
$\cA_0(V)\subset \cW(V)^{\text{Vect}(V, \omega_0)}$ and  $\cA_1(V)\subset \cW(V)^{\text{Vect}(V, \omega_1)}$.
\end{lemma}
\begin{thm} If $d = \dim V=2$, $\cW(V)^{\text{Vect}(V, \omega_0)}=\cA_0(V).$
\end{thm}
It was conjectured in \cite{S2} that for all $d$, $\cW(V)^{\text{Vect}(V, \omega_0)}=\cA_0(V)$ and $\cW(V)^{\text{Vect}(V, \omega_1)}=\cA_1(V).$ In this section, we will prove this conjecture.

\subsection{$\text{Vect}_0(V, \omega_i)[t]$-invariants}

Let $\mathfrak g_0=\text{Vect}_0(V, \omega_i)$.
Let $\mathfrak g_0[t]=\oplus_{n\geq 0} \mathfrak g_0 t^n$ be the Lie algebra given by
$$[g_i t^i, g_jt^j]=[g_i,g_j]t^{i+j}, \quad \text{for} \quad g_i,g_j\in \mathfrak g_0. $$
The action of $\mathfrak g_0$ on $V$ induces an action of $\mathfrak g_0[t]$ on $SW(V)$, which is given by
\begin{eqnarray*}
gt^n \beta^{x'_i}_{(-k)}&=&\beta^{g x'_i}_{(-k+n)}, \ n<k, \quad\quad \quad gt^n \beta^{gx'_i}_{(-k)}=0, \ n\geq k, \\
gt^n b^{x'_i}_{(-k)}&=&b^{g x'_i}_{(-k+n)}, \ n<k, \quad  \quad\quad gt^n b^{g x'_i}_{(-k)}=0, \ n\geq k,  \\
gt^n c^{x_i}_{(-k)}&=&c^{g x_i}_{(-k+n)}, \ n<k, \quad  \quad \quad gt^n c^{gx_i}_{(-k)}=0, \ n\geq k,  \\
gt^n \gamma^{x_i}_{(-k)}&=&\gamma^{g x_i}_{(-k+n)}, \ n<k-1, \quad  gt^n \gamma^{gx_i}_{(-k)}=0, \ n\geq k-1.
\end{eqnarray*}

Note that $SW(V)$ is a ring with a derivation $\partial$, given by
$\partial P_{(-k)}=k P_{(-k-1)}$, for $P=\beta^{x'_i}, b^{x'_i}, c^{x_i}$ and $\alpha^{x_i}$. For $R\subset SW(V)$, let $R^{\mathfrak g_0[t]}$ denote the subspace of ${\mathfrak g_0[t]}$-invariants in $R$.

As preparation for the next lemma, we recall the following results from \cite{LSSII,LSII,LSIII}. Given an algebraic group $G$ over $\mathbb{C}$ and a finite-dimensional $G$-module $V$, the arc space $J_{\infty}(G)$ is an algebraic group which acts on the arc space $J_{\infty}(V)$. The quotient morphism $V\to V/\!\!/G$ induces a morphism $J_\infty(V)\to J_\infty(V/\!\!/G)$, so we have a morphism 
\begin{equation} J_\infty(V)/\!\!/J_\infty(G)\to J_\infty(V/\!\!/G).\end{equation}
In particular, we have a ring homomorphism
\begin{equation} \label{invariantringmap} \mathbb{C}[J_{\infty}(V/\!\!/G)] \rightarrow \mathbb{C}[J_{\infty}(V)]^{J_{\infty}(G)}.\end{equation} If $V/\!\!/G$ is smooth or a complete intersection, and $\mathbb{C}[V]$ has no nontrivial one-dimensional $G$-invariant subspaces, it was shown in \cite{LSSI} that \eqref{invariantringmap} is an isomorphism, although in general it is neither injective nor surjective. If \eqref{invariantringmap} is surjective, it follows that $\mathbb{C}[J_{\infty}(V)]^{J_{\infty}(G)}$ is generated as a differential algebra by the subalgebra $\mathbb{C}[V]^{G}$.

More explicitly, let $V_j\cong V$ for $j\geq 0$, and fix a basis $\{x_{1,j},\dots, x_{n,j}\}$ for $V_j$. Let $S = \mathbb{C}[\bigoplus_{j\geq 0} V_j]$. The map $\mathbb{C}[J_{\infty}(V)] = \mathbb{C}[x_1^{(j)},\dots, x_n^{(j)} |\ j\geq 0]   \rightarrow S$ sending $x_i^{(j)}\mapsto x_{i,j}$ is an isomorphism of differential algebras, where the differential $\partial$ on $S$ is given by $\partial(x_{i,j}) = (j+1)x_{i,j+1}$. In particular, the subalgebra $S_0 = \mathbb{C}[V_0]$ generates $S$ as a differential algebra.

For $j\geq 0$, let $\tilde{V}_j\cong V$ and let $L= \bigwedge \bigoplus_{j\geq 0} \tilde{V}_j$. Fix a basis $\{y_{1,j},\dots, y_{n,j}\}$ for $\tilde{V}^*_j$ and extend the differential on $S$ to an even differential $\partial$ on $S\otimes L$, defined on generators by $\partial(y_{i,j}) = (j+1)y_{i,j+1}$. There is an action of $J_{\infty}(G)$ on $S\otimes L$, and we may consider the invariant ring $(S\otimes L)^{J_{\infty}(G)}$. 
Let $L_0 = \bigwedge (\tilde{V}_0)\subset L$, and let $\bra (S_0\otimes L_0)^G\ket$ be the differential algebra generated by $(S_0\otimes L_0)^G$, which lies in $(S\otimes L)^{J_{\infty}(G)}$.

Since $G$ acts on the direct sum $V^{\oplus k}$ of $k$ copies of $V$, we have a map 
\begin{equation} \label{kcopiesofv}\mathbb{C}[J_{\infty}(V^{\oplus k} /\!\!/ G)] \rightarrow \mathbb{C}[J_{\infty}(V^{\oplus k})]^{J_{\infty}(G)}.\end{equation} 
\begin{thm} \label{oddgen} (\cite[Thm. 7.1]{LSSII}) Suppose that \eqref{kcopiesofv} is an isomorphism for all $k\geq 1$. Then $(S\otimes L)^{J_{\infty}(G)} =\bra (S_0\otimes L_0)^G\ket$. \end{thm} In fact, under the above hypothesis, all differential algebraic relations in $(S\otimes L)^{J_{\infty}(G)}$ are consequences of relations among the generators of $(S_0\otimes L_0)^G$, and their derivatives \cite[Thm. 3.1 (2)]{LSIV}), but this stronger fact will not be needed in this paper. By \cite[Cor. 1.5]{LSII}, the hypothesis of Theorem \ref{oddgen} is satisfied in the case $G = Sp_{2d}$ and $V= \mathbb{C}^{2d}$.

In the case $G = SL_d$ and $V= \mathbb{C}^d \oplus (\mathbb{C}^d)^*$, this hypothesis is not satisfied since \eqref{kcopiesofv} is surjective for all $k$ but fails to be injective when $k \geq d + 3$; see \cite[Thm. 1.2]{LSIII}. However, the surjectivity of \eqref{kcopiesofv} for all $k$ in this case is enough for our purposes, due to the following:
\begin{thm} \label{oddgenweaker} (\cite[Thm. 3.1 (1)]{LSIV}) Suppose that \eqref{kcopiesofv} is surjective for all $k\geq 1$. Then $(S\otimes L)^{G_{\infty}}=\bra (S_0\otimes L_0)^G\ket$. \end{thm} 
This applies to the case of $G = SL_d$ and $V= \mathbb{C}^d \oplus (\mathbb{C}^d)^*$. Note that if \eqref{kcopiesofv} fails to be injective for some $k$, it need not be the case that all differential algebraic relations in $(S\otimes L)^{G_{\infty}}$ are consequences of relations in $(S_0\otimes L_0)^G$ and their derivatives, but this does not affect our results.

\begin{lemma}\label{lemma:gtinvariant} Recall the isomorphism $\pi:SW(V)\to \bW(V)$ given by \eqref{eqn:isopi}.
\begin{enumerate} 
\item If $\mathfrak g_0= \text{Vect}_0(V, \omega_0)$, as a ring with a derivation $\partial$, $SW(V)^{\mathfrak g_0[t]}$ is generated by $$\pi^{-1}(Q(z)), \pi^{-1}(L(z)), \pi^{-1}(G(z)), \pi^{-1}(J(z)),\pi^{-1}(E(z)), \pi^{-1}(B(z)),\pi^{-1}(C(z)), \pi^{-1}(D(z)).$$
\item If $\mathfrak g_0= \text{Vect}_0(V, \omega_1)$, as a ring with a derivation $\partial$, $SW(V)^{\mathfrak g_0[t]}$ is generated by $$\pi^{-1}(Q(z)), \pi^{-1}(L(z)), \pi^{-1}(G(z)), \pi^{-1}(J(z)),\pi^{-1}(E'(z)), \pi^{-1}(B'(z)),\pi^{-1}(C'(z)), \pi^{-1}(D'(z)).$$
\end{enumerate}
\end{lemma}
\begin{proof} For the first statement, $\mathfrak{g}_0 = \mathfrak{sl}_d$, and we have an isomorphism of $\mathfrak{sl}_d[t]$-modules 
$$SW(V) \cong \mathbb{C}[J_{\infty} (V \oplus V^*)] \otimes \bigwedge \big(\oplus_{j\geq 0} \big(V_j \oplus V^*_j)\big)= S \otimes L,$$ where $V = \mathbb{C}^d$. Under the linear isomorphism \eqref{eqn:isopi}, the above fields correspond to the generators of the subalgebra $(S_0 \otimes L_0)^{SL_d}$, which by Theorem \ref{oddgenweaker} generate $(S \otimes L)^{J_{\infty}(SL_d)} = (S \otimes L)^{\mathfrak{sl}_d[t]}$ as a differential algebra. 

The second statement is proven in the same way using Theorem \ref{oddgen}, since $\mathfrak{g}_0 = \mathfrak{sp}_{2d}$ and we have an isomorphism of $\mathfrak{sp}_{2d}[t]$-modules 
$$SW(V) \cong \mathbb{C}[J_{\infty} (V)] \otimes \bigwedge \big(\oplus_{j\geq 0} V_j\big)= S \otimes L,$$ where $V = \mathbb{C}^{2d}$. Then the above fields correspond to the generators of $(S_0 \otimes L_0)^{Sp_{2d}}$, and hence generate $SW(V)^{\mathfrak{sp}_{2d}[t]}$ as a differential algebra. \end{proof}

\subsection{$\text{Vect}(V, \omega_i)$-invariants}
Let $SW_n(V)$ be the linear subspace of $SW(V)$ which is spanned by the monomials of $\gamma_{(i-1)},\beta_{(i)},b_{(i)},c_{(i)}$, $i<0$  with the property that the number of $c$ in the monomial plus double of the number of $\gamma$ in the monomial is
$n$. We then have the grading
$$SW(V)=\oplus_{n\geq 0}SW_n(V).$$ Since the action of $\mathfrak g_0[t]$ on $SW(V)$ preserves $SW_n(V)$, $SW(V)^{\mathfrak g_0[t]}=\oplus_{n\geq 0}SW_n(V)^{\mathfrak g_0[t]}$.

\begin{lemma}\label{lemma: ginvariant}
Let $a\in \cW(V)^{\text{Vect}(V, \omega_i)}$ be homogeneous with respect to conformal weight. Then
\begin{enumerate}
\item $a \in \cW_+(V)$. In particular, we may write $a=\pi(a_k+a_{k-1}+\cdots)$ where $\pi$ is given by \eqref{eqn:isopi}, and $a_n\in SW_n(V)$.
\item The leading term $a_k$ is $\text{Vect}_0(V, \omega_i)[t]$-invariant.
\end{enumerate}
\end{lemma}
\begin{proof} It is easy to see that $\text{Vect}_{-1}(V, \omega_i)=\text{Vect}_{-1}(V)$. So for any $a\in \cW(V)^{\text{Vect}(V, \omega_i)}$, $\cL(\frac{\partial}{\partial x_j}) a =\beta^{x'_j}_{(0)}a=0$ for any $1\leq j\leq d$. Therefore $\gamma^{x_i}_{(-1)}$ does not appear in $a$, so that $a\in \cW_+(V)$. Since $a$ has fixed conformal weight, it is apparent that it has the form $a=\pi(a_k+a_{k-1}+\cdots)$ with $a_n\in SW_n(V)$. This proves (1).
	
	Next, let $ \mathfrak g_j=\text{Vect}_j(V, \omega_i)$.
	 It is easy to see that $\pi$ is $\mathfrak g_0$-equivariant.  So $a_k$ is $\mathfrak g_0$-invariant. Let $v_1=x_1^2\frac{\partial}{\partial x_2}\in \mathfrak g_1$ and $g_1=x_1\frac{\partial}{\partial x_2}\in \mathfrak g_0$.
Let $$K_1=\sum_{l\geq 1}\gamma_{(-l-1)}^{x_1}g_1t^l.$$ We have
$$0=\mathcal L(v_1)a=(2:(:\gamma^{x_1}c^{x_1}:)b^{x'_2}:_{(0)}+:(:\gamma^{x_1}\gamma^{x_1}:)\beta^{x'_2}:_{(0)})\pi(a_k+a_{k-1}+\cdots) .$$
Consider the homogeneous component $SW_{k+2}(V)$:
$$0=(\mathcal L(v_1)a)_{k+2}=2\sum_{l =1}^\infty \gamma^{x_1}_{(-l-1)} g_1t^{l}a_k=2K_1 a_k.$$
Similarly, let $v_0=x_1^2\frac{\partial}{\partial x_1}-2x_1x_2\frac{\partial}{\partial x_2}\in \mathfrak g_1$ and $g_0=x_1\frac{\partial}{\partial x_1}-x_2\frac{\partial}{\partial x_2}\in \mathfrak g_0$.
Let $$K_0=\sum_{l\geq 1}\gamma_{(-l-1)}^{x_1}g_0t^l-\sum_{l\geq 1}\gamma_{(-l-1)}^{x_2}g_1t^l.$$
We have $K_0 a_k=0$.

Inductively, let $K_n=[K_0, K_{n-1}]$, for $n\geq 2$, then $K_n a_k=0$.
$$[K_0, \gamma_{(-l)}^{x_1}]=\sum_{s=2}^{l-2} \gamma_{(-s)}^{x_1}\gamma_{(-l+s)}^{x_1}.$$
$$[K_0, g_1t^j]=2 \sum_{s\geq 1} \gamma^{x_1}_{(-s-1)}g_1t^{j+s}.$$
So inductively, we obtain
$$K_n=\sum_{l_i\geq 1}c_{l_1,\dots,l_n} (\prod_{i=1}^{n} \gamma_{-l_i-1}^{x_1})g_1t^{l_1+\cdots +l_n}.$$
Here $c_{l_1,\dots,l_n}$ are positive numbers.
When $l$ is large enough, $g_1t^{l} a_k=0$. Let $L$ be the largest number such that $g_1t^{L}a_k\neq 0$. If $L\geq 1$
then $$0=K_La_k= c_{1,\dots,1}(\gamma_{-2}^{x_1}))^Lg_1t^{L}a_k\neq 0.$$
So $g_1t \, a_k =0$. Since $\mathfrak g_0$ is a simple Lie algebra, $\mathfrak g_0[t]$ is generated by $\mathfrak g_0$ and $g_1t$. So $a_k$ is $\mathfrak g_0[t]$-invariant.
\end{proof}

\begin{thm}\label{thm:invariant} $\cW(V)^{\text{Vect}(V, \omega_0)}=\cA_0(V);$ $\cW(V)^{\text{Vect}(V, \omega_1)}=\cA_1(V)$. \end{thm}
\begin{proof}For the first equation, let $\mathfrak g_0= \text{Vect}_0(V, \omega_0)$. 
	By Lemma \ref{lemma:gtinvariant},
 any $a_k\in SW_k(V)^{\mathfrak g_0[t]}$ can be represented as a polynomial in $\partial^l \pi^{-1}(Q(z))$, $\partial^l\pi^{-1}(L(z))$, $\partial^l\pi^{-1}(G(z))$,$\partial^l \pi^{-1}(J(z))$, $\partial^l \pi^{-1}(E(z))$, $\partial^l \pi^{-1}(B(z))$, $\partial^l\pi^{-1}(C(z))$, $\partial^l\pi^{-1}(D(z))$. Let $b$ be the corresponding normally ordered polynomial in $\partial^lL(z)$, $\partial^lG(z)$, $\partial^l J(z)$, $\partial^l E(z)$, $\partial^l B(z)$, $\partial^lC(z)$, $\partial^lD(z)$. We have
$\pi^{-1}(b)=b_k+b_{k-1}+\cdots$,  $b_n\in SW_n(V)$ with $b_k=a_k$.

If $a\in \mathcal W(V)^{\text{Vect}(V, \omega_0)}$, $\pi^{-1}(a)=a_k+a_{k-1}+\cdots.$ By Lemma \ref{lemma: ginvariant}, $a_k\in SW_k(V)^{\mathfrak g_0[t]}$. So there is a $b\in \cA_0(V)$ with $\pi^{-1}(b)=b_k+b_{k-1}+\cdots$ and $a_k=b_k$.
Thus $\pi^{-1}(a-b)=(a_{k-1}-b_{k-1})+\cdots$ and $a-b$ is $\text{Vect}(V, \omega_0)$-invariant. By induction on $k$, we conclude $a\in \cA_0(V)$. So $\cW(V)^{\text{Vect}(V, \omega_0)}=\cA_0(V)$.

The proof for the second equation is similar.
\end{proof}

\section{Chiral de Rham complex}\label{sec:three}
Let $\mathcal W=\mathcal W(\mathbb C^d)$ and $x_1',\dots,x_d'$ be a standard basis of $\mathbb C^d$. Then $\cW$ has strong generators $\beta^i=\beta^{x'_i}$, $b^i=b^{x'_i}$, $\gamma^i=\gamma^{x_i}$ and $c^i=c^{x_i}$, and is a free $\mathbb C[\gamma^1_{(-1)},\dots, \gamma^d_{(-1)}]$ module. If $X$ is a complex manifold and $(U,\gamma^1,\dots, \gamma^d)$ is a complex coordinate system of $X$,
$\mathcal O(U)$ is a $\mathbb C[\gamma^1_{(-1)},\dots, \gamma^d_{(-1)}]$-module
by identifying the action of $\gamma^i_{(-1)}$ with  the product of $\gamma^i$. 
The chiral de Rham complex $\Omega_X^{\text{ch}}$ is a sheaf of vertex algebras on $X$ whose algebra of sections $\Omega_X^{\text{ch}}(U)$ is given by $$\Omega_X^{\text{ch}}(U)=\cW \otimes_{\mathbb C[\gamma^1_{(-1)},\dots, \gamma^d_{(-1)}]}\mathcal O(U).$$
In particular, $\Omega_X^{\text{ch}}(U)$ is the vertex algebra with strong generators $\beta^i(z), b^i(z), c^i(z)$ and $f(z)$, $f\in \mathcal O(U)$. The nontrivial OPEs among these generators are
$$\beta^i(z)  f(w)\sim \frac{\partial f}{\partial \gamma^i}(w)(z-w)^{-1},\quad b^i(z) c^j(w)\sim \delta^i_j (z-w)^{-1},$$ as well as the normally ordered product relations $$:f(z)g(z):\ =fg(z), \text{ for }  f, g \in \mathcal O(U).$$

Let $\tilde \gamma^1,\dots, \tilde \gamma^d$ be another set of coordinates on $U$, with
$$\tilde \gamma^i=f^i(\gamma^1,\dots, \gamma^d), \quad \gamma^i=g^i(\tilde \gamma^1,\dots, \tilde \gamma^d).$$
We have the following coordinate change equations:
\begin{align}\label{chi.coo}
	\partial \tilde \gamma^i(z)&=\sum_{j=1}^d:\frac{\partial f^i}{\partial \gamma^j}(z)\partial \gamma^j(z):\,, \nonumber \\
	\tilde b^i(z)&=\sum_{j=1}^d :\frac{\partial g^j}{\partial \tilde \gamma^i}(f(\gamma))(z)b^j(z): \nonumber\,, \\
	\tilde c^i(z)&=\sum_{j=1}^d:\frac{\partial f^i}{\partial \gamma^j}(z)c^j(z):\,,\\
	\tilde \beta^i(z)&=\sum_{j=1}^d :\frac{\partial g^j}{\partial \tilde \gamma^i}(f(\gamma))(z)\beta^j(z):
	+\sum_{k=1}^d:(:\frac{\partial}{\partial \gamma^k}(\frac{\partial g^j}{\partial \tilde \gamma^i}(f(\gamma)))(z)c^k(z):)b^j(z):\,.\nonumber
\end{align}

\subsection{Global sections}
There are four sections $Q(z), L(z), J(z)$ and $G(z)$ from (\ref{eqn:globlesection0}) in $\Omega_X^{\text{ch}}(U)$. For a general complex manifold $X$,
$L(z)$ and $G(z)$ are globally defined and have the same form in any local coordinate system. The fields $Q(z)$ and $J(z)$ are globally defined if and only if the first Chern class $c_1(TX)$ vanishes, but their zero modes $Q_{(0)}$ and $J_{(0)}$, are always globally defined \cite{MSV}. The operators $L_{(1)}$ and $J_{(0)}$ give $\Omega_X^{\text{ch}}$ a $\mathbb Z_{\geq 0}\times\mathbb Z $-grading by conformal weight $k$ and degree $l$, respectively.
$$\Omega_X^{\text{ch}}=\bigoplus_{k,l} \Omega_X^{\text{ch}}[k,l].$$ Note that the zero mode $Q_{(0)}$ of $Q(z)$ is the chiral de Rham differential, and it preserves conformal weight and raises the degree by one.

If $X$ is a Calabi-Yau manifold with a nowhere vanishing holomorphic $d$-form $w_0$, let $(U,\gamma_1,\dots,\gamma_d)$ be a coordinate system on $X$ such that locally,
$$w_0|_U=d\gamma^1\cdots d\gamma^d.$$
The eight sections $Q(z), L(z), J(z), G(z),B(z),C(z),D(z)$ and $E(z)$
from \eqref{eqn:globlesection0} and \eqref{eqn:globlesection1} in $\Omega_X^{\text{ch}}(U)$ are globally defined on $X$ \cite{EHKZ}.

If $X$ is a hyperk\"ahler manifold with holomorphic symplectic form $w_1$, let $(U,\gamma_1,\dots,\gamma_d)$ be a coordinate system on $X$ such that
locally, $$w_1|_U=\sum_{i=1}^{\frac d 2} d\gamma^{2i-1}\wedge d\gamma^{2i}.$$ Then
the eight sections $Q(z), L(z), J(z), G(z),B'(z),C'(z),D'(z)$ and $E'(z)$
from \eqref{eqn:globlesection0} and \eqref{eqn:globlesection2} in $\Omega_X^{\text{ch}}(U)$ are globally defined on $X$ \cite{BHS}.
\begin{defn} If $X$ is a Calabi-Yau manifold with a nowhere vanishing holomorphic $d$-form, let $\cA_0(X)$ be the vertex algebra which is strongly generated by the eight global sections given by $Q(z)$, $L(z)$, $J(z)$, $G(z)$, $B(z)$, $C(z)$, $D(z)$ and $E(z)$ on $X$.\\
	 If $X$ is a hyperk\"ahler manifold, let $\cA_1(X)$ be the vertex algebra which is strongly generated by the eight global sections given by $Q(z), L(z), J(z), G(z),B'(z),C'(z),D'(z)$ and $E'(z)$ on $X$.
\end{defn}

The following theorem was proven in \cite{S2}.
\begin{thm}\label{thm:global}If $X$ is a $d$-dimensional compact K\"ahler manifold with holonomy group $G=SU(d)$ and $w_0$ is a nowhere vanishing holomorphic $d$-form, then 
	$$\Gamma (X,\Omega_X^{\text{ch}})\cong \cW_+( T_xX)^{\text{Vect}(T_xX, w_0|_x)}.$$
	If $X$ is a $d$-dimensional compact K\"ahler manifold with holonomy group $G=Sp(\frac d 2)$ and $w_1$ is a holomorphic symplectic form, then the space of global section of $\Omega_X^{\text{ch}}$ 
	$$\Gamma (X,\Omega_X^{\text{ch}})\cong  \cW_+(T_xX)^{\text{Vect}( T_xX,w_1|_x)}.$$
\end{thm}
Thus we have 
\begin{thm}\label{thm:globalfinal}
If $X$ is a $d$-dimensional compact K\"ahler manifold with holonomy group $G=SU(d)$, then
$$\cA_0(X)=\Gamma(X,\Omega_X^{\text{ch}})\cong \cA_0(\mathbb C^d).$$
If $X$ is a $d$-dimensional compact K\"ahler manifold with holonomy group $G=Sp(\frac d 2)$, then the eight global sections given by $Q(z), L(z), J(z), G(z),B'(z),C'(z),D'(z)$ and $E'(z)$ strongly generate
$$\cA_1(X)=\Gamma(X,\Omega_X^{\text{ch}})\cong\cA_1(\mathbb C^d).$$
\end{thm}
\begin{proof} If $X$ is a $d$-dimensional compact K\"ahler manifold with holonomy group $G=SU(d)$, there must be a nowhere vanishing holomorphic $d$-form $w_0$. By Theorem \ref{thm:global}, $\Gamma(X,\Omega_X^{\text{ch}})\cong \cW_+( T_xX)^{\text{Vect}(T_xX, w_0|_x)}$. By Theorem \ref{thm:invariant}, $\cW_+( T_xX)^{\text{Vect}(T_xX, w_0|_x)}$ is isomorphic to $\cA_0(\mathbb C^d)$.
So $\Gamma(X,\Omega_X^{\text{ch}})\cong \cA_0(\mathbb C^d).$ The isomorphism maps the global sections given by $Q(z)$, $L(z)$, $J(z)$, $G(z)$, $B(z)$, $C(z)$, $D(z)$ and $E(z)$  to $Q(z)$, $L(z)$, $J(z)$, $G(z)$, $B(z)$, $C(z)$, $D(z)$ and $E(z)$ themselves. So the eight global sections given by $Q(z)$, $L(z)$, $J(z)$, $G(z)$, $B(z)$, $C(z)$, $D(z)$ and $E(z)$ strongly generate $\Gamma(X,\Omega_X^{\text{ch}})\cong \cA_0(\mathbb C^d).$

Similarly, if $X$ is a $d$-dimensional compact K\"ahler manifold with holonomy group $G=Sp(\frac d s)$, there must be a holomorphic symplectic form $w_1$. Then by Theorem \ref{thm:global}, $\Gamma(X,\Omega_X^{\text{ch}})\cong \cW_+( T_xX)^{\text{Vect}(T_xX, w_0|_x)}$. By Theorem \ref{thm:invariant}, $\cW_+( T_xX)^{\text{Vect}(T_xX, w_0|_x)}$ is  isomorphic to $\cA_0(\mathbb C^d)$.
So $\Gamma(X,\Omega_X^{\text{ch}})\cong \cA_0(\mathbb C^d)$. The isomorphism maps the global sections given by $Q(z)$, $L(z)$, $J(z)$, $G(z)$, $B'(z)$, $C'(z)$, $D'(z)$ and $E'(z)$  to $Q(z)$, $L(z)$, $J(z)$, $G(z)$, $B'(z)$, $C'(z)$, $D'(z)$ and $E'(z)$ themselves. So the eight global sections given by $Q(z)$, $L(z)$, $J(z)$, $G(z)$, $B'(z)$, $C'(z)$, $D'(z)$ and $E'(z)$ strongly generate
$\Gamma(X,\Omega_X^{\text{ch}})\cong \cA_0(\mathbb C^d).$
\end{proof}

\subsection{Covering maps}
Let $X$ and $Y$ be compact complex manifolds and let $p: Y\to X$ be a covering map. 
By the definition of chiral de Rham complex, the inverse image sheaf $p^{-1} \Omega_X^{\text{ch}}=\Omega_Y^{\text{ch}}$. A global section of $\Omega_X^{\text{ch}}$ pulls back to a global section of $\Omega_Y^{\text{ch}}$.
Let $$p^*:
\Gamma(X,\Omega_X^{\text{ch}}) \to \Gamma(Y,\Omega_Y^{\text{ch}}) $$ be the pullback map.  If $p$ is an isomorphism, then  $p^*$ is clearly an isomorphism.

If $p$ is a finite normal covering map, let $G(Y,p)$ be its covering transformation group. 
For any $g\in G(y,p)$, the action of $g$ on $Y$, $\rho(g):Y\to Y$ induces an automorphism  $\rho(g)^*: \Gamma(Y,\Omega_Y^{\text{ch}}) \to \Gamma(Y,\Omega_Y^{\text{ch}})$. Let 
$\Gamma(Y,\Omega_Y^{\text{ch}})^{G(Y,p)}$ be the invariant subalgebra under the induced action of $G(Y,p)$.
\begin{prop}\label{prop:cover}
	$p^*$ induces an isomorphism of vertex algebras $\Gamma(X,\Omega_X^{\text{ch}}) \to \Gamma(Y,\Omega_Y^{\text{ch}})^{G(Y,p)}.
	$ 
\end{prop} 
\begin{proof} Obviously, $p^*$ is an injective morphism of vertex algebras. For any $g\in G(Y,p)$, $p\circ \rho(g)=p$. So $\rho(g)^*\circ p^*=p^*$. For any section $a\in \Gamma(X,\Omega_X^{\text{ch}})$,  $\rho(g)^*( p^*(a))=p^*(a)$, so $p^*(a)$ is $G(Y,p)$-invariant.
	
On the other hand, assume $p$ is an $n$-sheet covering map. There is an open cover $\{U_\alpha\}$ of $X$ such that each $p^{-1}(U_\alpha)$ is the disjoint union of open sets  $V_{\alpha,i}$ in $Y$, and $p|_{V_{\alpha,i}}:V_{\alpha,i}\to U_\alpha$ is an isomorphism. Let $\tilde a \in \Gamma(Y,\Omega_Y^{\text{ch}})^{G(Y,p)}$, and define $$a_\alpha=\frac 1 n \sum_{i=1}^n ((p|_{V_{\alpha,i}})^{-1})^* (\tilde a|_{V_{\alpha,i}}).$$ For another open set $U_\beta$ in the open cover, it is easy to see $a_\alpha|_{U_\alpha\cap U_\beta}=a_\beta|_{U_\alpha\cap U_\beta}$, so there is  an $a\in \Gamma(X,\Omega_X^{\text{ch}})$ with $a|_{U_\alpha}=a_\alpha$. It is easy to see that $p^*(a)=\tilde a$, since $\tilde a$ is $G(Y,p)$-invariant.
\end{proof}

\subsection{Global sections: general case}
For a compact Ricci-flat K\"ahler manifold, we have the following properties (Proposition 6.22, 6.23 in \cite{J}).
\begin{prop}\label{prop:ricciflat1}
	Let $X$ be a compact Ricci-flat K\"ahler manifold. Then $X$ admits a finite cover isomorphic to the product K\"ahler manifold
	$T^{2l}\times X_1\times X_2\cdots\times X_k$, where $T^{2l}$ is a flat K\"ahler torus and $X_j$ is a compact, simply connected, irreducible, Ricci-flat K\"ahler manifold for $j=1,\dots,k$.
\end{prop}
\begin{prop}\label{prop:ricciflat2}
	Let $X$ be a compact, simply-connected, irreducible, Ricci-flat K\"ahler manifold of dimension $d$. Then either $d\geq 2$ and its holonomy group is $SU(d)$, or $d\geq 4$ is even and its holonomy group is $Sp(\frac d 2)$. Conversely, if $X$ is a compact K\"ahler manifold and its holonomy group is $SU(d)$ or $Sp(\frac d 2)$, then $X$ is Ricci-flat and irreducible and $X$ has finite fundamental group.
\end{prop}
\begin{lemma}\label{lem:normalgroup} Let $G$ be a group and let $N\cong \mathbb Z^k$ be a subgroup with finite index in $G$. Then there is a subgroup $M$ of $N$ such that the index of $M$ in $N$ is finite, and $M$ is a normal subgroup of $G$.
\end{lemma}
\begin{proof} Let $N_1, \dots, N_l$ be all of the conjugate subgroups of $N$ in $G$, and let $M=\cap N_i$, so that $M$ is a normal subgroup of $G$. Since the index of $N_i$ in $G$ is finite, for any $g \in G$, there is an integer $m_i>0$ such that $g^{m_i}\in N_i$. Let $m=m_1\cdots m_l$. Then $g^m\in N_i$ for all $1\leq i\leq l$, so that $g^m \in M$. If $g_1,\dots, g_k$ are generators of $N$, there exist positive integers $m^1,\dots, m^k$ such that $g^{m^i}_i\in M$. Since $N$ is a free abelian group, the index of $M$ in $N$ is no more than $m^1\cdots m^k$. \end{proof}
	
\begin{prop}\label{prop:ricciflat3}
	The finite covering map in Proposition \ref{prop:ricciflat1} can be chosen to be a normal covering map.
\end{prop}
\begin{proof}
	Let $Y=T^{2l}\times X_1\times X_2\cdots\times X_k$ be the finite cover in Proposition \ref{prop:ricciflat1}, and let  $p: Y\to X$ be the covering map. It induces an injective morphism of fundamental groups $p_*: \pi_1(Y,y)\to \pi_1(X, x)$ for $x=p(y)$. Since each $X_i$ is simply connected,  $\pi_1(Y)\cong\pi_1(T^{2l}) \cong\mathbb Z^{2l}$. Since $p$ is a finite covering map, the index of $p_*(\pi_1(Y,y))$ in $\pi_1(X,x)$ is finite. By Lemma \ref{lem:normalgroup}, there is a finite index subgroup $M$ of $\pi_1(Y,y)$, such that $p_*(M)$ is a normal subgroup of $\pi_1(X,x)$. We have a covering map $p_1: Y\to Y$ (given by the covering map $T^{2l}\to T^{2l}$) with $p_{1*}(\pi_1(Y,y_1))=M\subset \pi_1(Y,y)$ for some $y_1\in p_1^{-1}(y)$.
Then the covering map 
$p\circ p_1: Y\to X$ is a finite normal covering map since $p_*\circ p_{1*}(\pi_1(Y,y_1))=p_*(M)$ is a normal subgroup of $\pi_1(X,x)$ with finite index.
\end{proof}

\begin{thm}\label{thm:globalfinal2}
Let $X$ be a compact Ricci-flat K\"ahler manifold. Let $Y=T^{2l}\times X_1\times X_2\cdots\times X_k$ be the finite cover of $X$ in Proposition \ref{prop:ricciflat1}. 
Let $p:Y\to X$ be the finite normal covering map in Proposition \ref{prop:ricciflat3}, and let $G(Y,p)$ be the covering transformation group. Then $$\Gamma(X,\Omega_X^{\text{ch}})\cong (\Gamma(T^{2l},\Omega_{T^{2l}}^{\text{ch}})\bigotimes(\otimes_{i=1}^n\cA_0(X_i))\bigotimes (\otimes_{i=n+1}^k\cA_1(X_i)))^{G(Y,p)}$$
through $p^*$.

\end{thm}
\begin{proof}
	Assume the dimension of $X_i$ is $d_i$.  By Proposition \ref{prop:ricciflat2}, we can assume  the holonomy group of $X_i$ is $SU(d_i)$ for $1\leq i\leq n$ and the holonomy group of $X_j$ is  $Sp(\frac {d_j} 2)$ for  $n<j\leq k$.
	By Theorem \ref{thm:globalfinal}, 
	$$\Gamma(Y,\Omega_Y^{\text{ch}})=\Gamma(T^{2l},\Omega_{T^{2l}}^{\text{ch}})\bigotimes(\otimes_{i=1}^n\cA_0(X_i))\bigotimes (\otimes_{i=n+1}^k\cA_1(X_i))$$
	By Proposition \ref{prop:cover}, $\Gamma(Y,\Omega_Y^{\text{ch}})\cong \Gamma(Y,\Omega_Y^{\text{ch}})^{G(Y,p)}$ through $p^*$. So $$\Gamma(X,\Omega_X^{\text{ch}})\cong (\Gamma(T^{2l},\Omega_{T^{2l}}^{\text{ch}})\bigotimes(\otimes_{i=1}^n\cA_0(X_i))\bigotimes (\otimes_{i=n+1}^k\cA_1(X_i)))^{G(Y,p)}$$
through $p^*$.	
\end{proof}
Since $\Gamma(T^{2l},\Omega_{T^{2l}}^{\text{ch}})\cong \cW_+(\mathbb C^{2l})$, the above theorem gives the space of global sections of chiral de Rham complex on compact Ricci-flat K\"ahler manifolds explicitly.

\section{Declarations}
The authors have no relevant financial or non-financial interests, conflicts, or competing interests to disclose. Data sharing is not applicable to this article as no datasets were generated or analyzed during the current study.

\end{document}